\documentclass[12pt]{amsart}
\usepackage{amsmath,amssymb,graphicx,verbatim,rotating}
\usepackage[colorlinks=true, pdfstartview=FitV, linkcolor=blue, citecolor=blue, urlcolor=blue]{hyperref}
\pdfoutput=1
\usepackage{mathpazo}

\newcommand{\tre}{\text{Re}}
\newcommand{\tim}{\text{Im}}
\newcommand{\hg}{\hat{g}}
\newcommand{\hh}{\hat{h}}
\newcommand{\intii}{\int_{-\infty}^\infty}

\theoremstyle{definition}
\newtheorem{theorem}{Theorem}
\newtheorem*{nonumtheorem}{Theorem}
\newtheorem{lemma}[theorem]{Lemma}
\newtheorem*{nonumlemma}{Lemma}
\newtheorem*{remark}{Remark}
\newtheorem*{remarks}{Remarks}

\begin{document}
\title[Repulsive behavior]{Repulsive behavior in an exceptional family}
\author{Jeffrey Stopple}
\begin{abstract}
The existence of a Landau-Siegel zero leads to the Deuring-Heilbronn phenomenon, here appearing in the $1$-level density in a  family of quadratic twists of a fixed genus character $L$-function.  We obtain explicit lower order terms describing the vertical distribution of the zeros, and realize the influence of the Landau-Siegel zero as a resonance phenomenon.
\end{abstract}
\email{stopple@math.ucsb.edu}\address{Mathematics Department, UC Santa Barbara, Santa Barbara CA 93106}

\maketitle

\subsection*{Introduction}
This paper would be called \emph{Statistical Deuring-Heilbronn phenomenon}, but for the fact that that title is already taken \cite{Ju}.  The Deuring-Heilbronn phenomenon is the influence of a Landau-Siegel zero of a quadratic Dirichlet $L$-function on both the vertical and horizontal distribution of the zeros of other $L$-functions.  Deuring, Heilbronn, and later Stark \cite{Deuring, Heilbronn, Stark-2, Stark1} obtained results by consideration of Epstein zeta functions.  Linnik's Theorem \cite[Theorem 5, p.172]{Montgomery} can be proven by Tur\'an's power sum method.  Pintz \cite{PintzIII, PintzIV} obtained results by elementary methods under very strong hypotheses on the class number.  Jutila, and Conrey and Iwaniec \cite{Ju, CI} used the approximate functional equation.

Modern approaches to the vertical distribution of the zeros of $L$-functions are motivated by considerations of Random Matrix Theory.  Usually, assuming the Generalized Riemann Hypothesis (GRH), the $1$-level density follows from the Explicit Formula so elegantly as to be an exercise in an undergraduate text \cite[Exercise 18.2.11]{MB}.

Here, we show the Deuring-Heilbronn phenomenon via the Explicit Formula, as the $1$-level density in a \lq family\rq\ of quadratic twists of a fixed genus character.  The question of the horizontal distribution of zeros is still very difficult;  good results for complex zeros, even in the presence of a Landau-Siegel zero, would give good lower bounds on the class number \cite{SZ}.
We have partial results in an Appendix.  When necessary, we assume Hypothesis H of Sarnak and Zaharescu (see below).  For the vertical distribution, we have two goals unrealized in the prior work cited above:
\begin{enumerate}
\item Obtain explicit lower order terms describing the vertical distribution of the zeros, in the presence of a Landau-Siegel zero.
\item  Realize the influence of the Landau-Siegel zero as a \lq resonance\rq\ phenomenon; see the remarks on page \pageref{resurgence}.
\end{enumerate}

\subsection*{Notation}
Suppose $-D<0$ is a fundamental discriminant.  Let $\psi$ be a genus character of the class group $\mathcal C(-D)$, corresponding to a factorization into fundamental discriminants (with opposite sign) $-D=d_1\cdot d_2$.  By the theorem of Kronecker, 
\[
L(s,\psi)=L(s,\chi_{d_1})L(s,\chi_{d_2}).
\]
Let $f$ be another fundamental discriminant such that $(f,-D)=1$.  Then the $L$-function of $\psi$ twisted by $\chi_f$ is just
\begin{equation}\label{Eq:twistedkronecker}
L(s,\psi\otimes\chi_f)=L(s,\chi_{fd_1})L(s,\chi_{fd_2}).
\end{equation}
Let $\mathcal{F}(X)$ denote the fundamental discriminants $f$ with $(f,D)=1$ and $X\le |f|<2X$, and let $X^*=\sharp \mathcal F(X)$.  If $L(s,\chi_{-D})$ has a Landau-Siegel zero $1-\delta$, we will call the family of $L(s,\psi\otimes\chi_f)$ as above exceptional.

We will make use of an even Schwartz test
function $g$ such that $\hg$ has compact support $\subset (-\sigma,\sigma)$.\footnote{This is stronger than what we actually need.  We require $\hg$ to have compact support as control of $\sigma$ is the fundamental problem.  So $g$ is smooth.  But $g(y)\ll 1/(1+y^2)$ is sufficient; we do not actually need $g$ to be rapidly decreasing.}  The use of the Burgess bound for character sums \cite[(12.57)]{IK} leads to the introduction of a parameter $\epsilon$.  We denote the trivial character modulo $D$ by $1_D$; without subscripts, $1$ and $\chi$ denote the trivial and nontrivial characters modulo $4$, respectively.  The Euler constant is $C$.  The important $S_{{\rm odd}}(\psi)$ is equal to (\ref{Eq:newsodd}).

\subsection*{Hypothesis H} Following the work of Sarnak and Zaharescu in \cite{SZ}, we sometimes invoke the following hypothesis on the zeros of the Dedekind zeta function $\zeta(s)L(s,\chi_{-D})$ and the zeros of $L(s,\psi\otimes\chi_f)$:  Except for a Landau-Siegel zero of the Dedekind zeta function at $\beta=1-\delta$ and at $\delta$, all the others are of the form $\rho=1/2+i \gamma$ with either
\begin{enumerate}
\item[(i)] $\rho$ on the critical line, i.e. $\gamma\in\mathbb R$ or,
\item[(ii)] $\rho$ is real, i.e. $i\gamma\in (-1/2,1/2)$.
\end{enumerate}
This is Hypothesis H of \cite{SZ} for $\zeta(s)$ and quadratic Dirichlet $L$-functions only, written with a notation to single out the Landau-Siegel zero.

\subsection*{Contents}

Here is an outline of what is in the subsequent sections:
\begin{enumerate}
\item[\S \ref{S:ef1ld}]  Review of the Explicit Formula and summary of later sections to develop the $1$-level density, unconditional with respect to both  $-D$ and $\sigma$.  Theorem \ref{T:Lot} gives the first term of an asymptotic expansion of the $1$-level density in powers of the average spacing of the zeros.
\item[\S \ref{S:0}] Assuming the exponent $e$ of the principal genus is small, Theorem \ref{T:algebra} gives results for $\sigma<2/e$ and $\log_D(X)<1/\sigma e-1/2$. 
Assuming the existence of a Landau-Siegel zero,  Theorem \ref{T:analysis} uses the Burgess bound on character sums to give results for $\sigma<4/3$ and $\log_D(X)<1/4$.
\item[\S \ref{S:1}] The contribution of the conductors to the Explicit Formula.
\item[\S \ref{S:2}] The contribution of the Gamma factors.
\item[\S \ref{S:even}] The contribution of the even powers of primes.
\item[\S \ref{S:odd}] The contribution of the odd powers of primes.
\item[\S \ref{S:Appendix}] The analogous Explicit Formula for the Dedekind zeta function $\zeta(s)L(s,\chi_{-D})$.
\item[\S \ref{S:Appendix2}] Appendix: Notes towards Hypothesis H.
\end{enumerate}

\section{Explicit formula and 1-level density}\label{S:ef1ld}
We write a generic zero in the critical strip of $L(s,\psi\otimes\chi_f)$ as $\rho=1/2+i\gamma$.  By (\ref{Eq:twistedkronecker}), the Explicit Formula for $L(s,\psi\otimes\chi_f)$ follows from that for quadratic Dirichlet $L$-functions.
\begin{nonumtheorem}[Explicit Formula for twisted genus
characters]\label{thm:oneldNT} Let $g$ be an even Schwartz test
function such that $\hg$ has compact support.    We have 
{\allowdisplaybreaks
\begin{multline}\label{Eq:4}
\frac{1}{X^*}\sum_{f \in \mathcal{F}(X)} \sum_{\gamma}
g\left(\gamma \frac{\log(\sqrt{D} X)}{2\pi}\right)=\\ 
\frac1{X^*} \sum_{f \in
\mathcal{F}(X)}  \frac{\log\left(Df^2/\pi^2\right)}{\log(\sqrt{D} X)}\intii g(\tau)d\tau+\\
\frac{1}{\log(\sqrt{D} X)}\intii g(\tau)\tre\Bigg[
\frac{\Gamma'}{\Gamma}\left(\frac14+\frac{i\pi
\tau}{\log(\sqrt{D} X)}\right)+
\frac{\Gamma'}{\Gamma}\left(\frac34+\frac{i\pi
\tau}{\log(\sqrt{D} X)}\right)
\Bigg]
d\tau \\
- \frac{2}{X^*} \sum_{f
\in \mathcal{F}(X)} \sum_{k=1}^\infty \sum_p \frac{(\chi_{fd_1}(p)^k+\chi_{fd_2}(p)^k)\log
(p)}{p^{k/2}\log(\sqrt{D} X)}\ \hg\left(\frac{\log (p^k)}{\log(\sqrt{D} X)}\right).
\end{multline}
}
\end{nonumtheorem}
The $L$-functions $L(s,\psi\otimes\chi_f)$ have conductor $Df^2\approx DX^2$, but are not primitive.  Since they factor (\ref{Eq:twistedkronecker}) as the product of two primitive $L$-functions of conductor $\approx \sqrt{D}X$,  the natural scale for the zeros is $\log(\sqrt{D}X)/2\pi$.

The discriminants $fd_1$ and $fd_2$ have opposite sign, so we get both possibilities for the $\Gamma^\prime/\Gamma$ term.
We analyze the first two lines on the right side above in \S \ref{S:1} and \S \ref{S:2}.

As for the sum over primes, because all the characters are quadratic, the analysis splits depending on
whether or not $k$ is even. Set 
\[
S_{{\rm odd}}(\psi) = 
- \frac{2}{X^*} \sum_{f
\in \mathcal{F}(X)} \sum_{l=0}^\infty \sum_p \frac{(\chi_{fd_1}(p)+\chi_{fd_2}(p))\log
p}{p^{(2l+1)/2}\log(\sqrt{D}X)}\ \hg\left(\frac{\log p^{2l+1}}{\log(\sqrt{D}X)}\right).
\]
and
\[
 S_{{\rm even}}  = 
- \frac{2}{X^*} \sum_{f
\in \mathcal{F}(X)} \sum_{l=1}^\infty \sum_p \frac{(\chi_{fd_1}(p)^2+\chi_{fd_2}(p)^2)\log
p}{p^l\log(\sqrt{D}X)}\ \hg\left(\frac{\log p^{2l}}{\log(\sqrt{D}X)}\right).
\]
We will see that there is no dependence on $\psi$ for the even powers.
In fact, for these even powers, we observe 
\[
\chi_{fd_1}(p)^2+\chi_{fd_2}(p)^2=
\begin{cases}
\chi_{d_1}(p)^2+\chi_{d_2}(p)^2&\text{ if }(f,p)=1,\\
\qquad \qquad 0&\text{ if } p|f.
\end{cases}
\]

We rewrite the $0$ as
\[
0=\left(\chi_{d_1}(p)^2+\chi_{d_2}(p)^2\right)-\left(\chi_{d_1}(p)^2+\chi_{d_2}(p)^2\right)
\]
and break $S_{{\rm even}}$ into two terms:
\begin{equation}\label{Eq:seven1}
S_{{\rm even};1} =
-2\sum_{\ell=1}^\infty \sum_p \frac{\left(\chi_{d_1}(p)^2+\chi_{d_2}(p)^2\right)\log
p}{p^\ell \log(\sqrt{D}X)}\ \hg\left(\frac{\log p^{2\ell}}{\log(\sqrt{D}X)}\right)
\end{equation}
where we simplified $(\sum_{f
\in \mathcal{F}(X)}1)/X^*=1$, and
\begin{equation}\label{Eq:seven2}
S_{{\rm even};2}=
\frac4{X^\ast} \sum_{f\in \mathcal{F}(X)}
\sum_{\ell=1}^\infty \sum_{p|f} \frac{\log p}{p^\ell \log(\sqrt{D}X)}\
\hg\left(\frac{\log p^{2\ell}}{\log(\sqrt{D}X)}\right),
\end{equation}
where we simplified with
\[
\chi_{d_1}(p)^2+\chi_{d_2}(p)^2=2\quad\text{ if } p|f,\text{ since }(f,d_1d_2)=1.
\]
Observe that
\[
\left(\chi_{d_1}(p)^2+\chi_{d_2}(p)^2\right)\log(p)=\begin{cases}2\log(p)&\text{ if } (p,D)=1\\
\log(p)&\text{ if }  p|D.
\end{cases}
\]
We deal with the even powers in \S \ref{S:even}.  The calculations are the same as those in \cite{Miller}; but we exercise care to avoid making any hypothesis about the support of $\hg$ for as long as possible.  Constants implied by $O(\ )$ statements should be universal.

The odd powers of primes are more interesting.  If $p$ is inert, then $\chi_{d_1d_2}(p)=\chi_{-D}(p)=-1$, so
\[
\chi_{d_1}(p) = -\chi_{d_2}(p),
\]
and consequently, for these primes
\[
\chi_{fd_1}(p)+\chi_{fd_2}(p)=\chi_f(p)\left(\chi_{d_1}(p)+\chi_{d_2}(p)\right)=0.
\]
The contribution to $S_{{\rm odd}}(\psi)$ of those $p$ which are inert is $0$.  
For those $p$ which split, $\chi_{d_1d_2}(p)=\chi_{-D}(p)=1$, so
\[
\chi_{d_1}(p) = \chi_{d_2}(p)\overset{\text{def.}}=\psi(Q),\text{ where the form } Q(x,y)\text{ represents }p.
\]
and consequently, for these primes
\[
\chi_{fd_1}(p)+\chi_{fd_2}(p)=2\chi_f(p)\psi(Q).
\]
Similarly, if $p|D$ then $p$ divides exactly one of $d_1$ and $d_2$, and the character corresponding to the factor prime to $p$ again defines the value of the genus character $\psi$ so that in this case
\[
\chi_{fd_1}(p)+\chi_{fd_2}(p)=\chi_f(p)\psi(Q).
\]
Thus we can rewrite
\begin{multline}\label{Eq:newsodd}
S_{{\rm odd}}(\psi) = 
- \frac{2}{X^*}  \sum_{l=0}^\infty \sum_p \frac{\lambda(p)\psi(Q)\log
p}{p^{(2l+1)/2}\log (\sqrt{D}X)}\times\\
 \hg\left(\frac{\log p^{2l+1}}{\log (\sqrt{D}X)}\right)\sum_{f
\in \mathcal{F}(X)}\chi_f(p).
\end{multline}
where $p$ is represented by the form $Q(x,y)$, and
\[
\lambda(p)\overset{\text{def.}}=1+\chi_{-D}(p)=
\begin{cases}
2&\text{ if }p\text{ splits},\\
1&\text{ if } p|D\\
0&\text{ if }p\text{ is inert}.
\end{cases}
\]

Combining the results of \S \ref{S:1}, \S \ref{S:2}, and \S \ref{S:even} we get the following, unconditionally with respect to $-D$ and also with respect to $\sigma$ such that the support of $\hg \subset (-\sigma,\sigma)$.
\begin{nonumtheorem}[$1$-level density for twists of a genus character]  We define
\begin{gather*}
\zeta_D(s)=\frac{\zeta(s)L(s,1_D)}{\zeta(s+1)^2}=\frac{\zeta(s)^2\prod_{q|D}(1-q^{-s})}{\zeta(s+1)^2},\quad\text{and}\\
\text{Rem}(r) =\sum_p \frac{ p\log(p)\left(1-p^{2 r}\right)}{(p+1) \left(1-p^{2 r+1}\right) \left(1-p^{2r+2}\right)}.
\end{gather*}
We have
{\allowdisplaybreaks
\begin{multline}\label{Eq:5}  
\frac1{X^\ast}\sum_{f\in \mathcal{F}(X)} \sum_{\gamma}
g\left(\gamma \cdot \frac{\log (\sqrt{D}X)}{2\pi}\right)  =\\
\left(2+\frac{2\log(4/\pi e)}{\log(\sqrt{D}X)}\right) \hg(0) -g(0) +\\
\frac{1}{\log(\sqrt{D} X)}\intii g(\tau)\tre\Bigg[
\frac{\Gamma'}{\Gamma}\left(\frac14+\frac{i\pi
\tau}{\log(\sqrt{D} X)}\right)+
\frac{\Gamma'}{\Gamma}\left(\frac34+\frac{i\pi
\tau}{\log(\sqrt{D} X)}\right)
\Bigg]
d\tau+\\
\frac{2}{\log (\sqrt{D}X)}\intii g(\tau)\tre\Bigg[
\frac{\zeta_D^\prime}{\zeta_D}\left(1+\frac{4\pi i \tau}{\log (\sqrt{D}X)}\right) \Bigg]d\tau+\\
\frac{4}{\log (\sqrt{D}X)}\intii g(\tau)\tre\Bigg[
\text{Rem}\left(\frac{2\pi i \tau}{\log
(\sqrt{D}X)}\right) \Bigg]d\tau+\\
S_{{\rm odd}}(\psi)+O\left(\frac{\max \hg\cdot\log(\omega(D))}{X^{1/2}\log(\sqrt{D}X)}\right).
\end{multline} 
}
\end{nonumtheorem}

\begin{remarks}
\begin{enumerate}
\item\label{resurgence} Bogomolny and Keating \cite{BoKe} were the first to observe that 
$(\zeta^\prime(s)/\zeta(s))^\prime$
similarly appears in the pair correlation for the Riemann zeros.  Berry and Keating \cite{BK} wrote in that context
\begin{quote}
\emph{\lq\lq The appearance of $\zeta(s)$ indicates an astonishing resurgence property of the zeros: in the pair correlation of high Riemann zeros, the low Riemann zeros appear as resonances.\rq\rq}
\end{quote}
The appearance of $\zeta^\prime(s)/\zeta(s)$ in the $1$-level density shows a resonance phenomenon in conductor aspect.  This was first noticed by Conrey and Snaith in \cite{CS}.

Figure \ref{Fig:1} shows a graph (in red) of the real part of $\zeta_D^\prime/\zeta_D(1+2it)$ for $-D=-1411=-83\cdot 17$ and $0\le t\le 10$.     
We know \cite[Theorem 9.6(A)]{Tit} that
\[
\frac{\zeta^\prime(s)}{\zeta(s)}=\sum_{|t-\gamma|\le1}\frac{1}{s-\rho}+O(\log(t)),
\]
so up to a small error, the logarithmic derivative at $s=1+2it$ is determined by the nearby zeros $\rho$, and is positive near such a zero.  
Unlabeled but clearly visible in Figure \ref{Fig:1} is the contribution of the nearby pole when $2t=14.134725\ldots$
One can also see in Figure \ref{Fig:1} the contribution of the periodic terms arising from the $q|D$ with periods $2\pi/\log(17)\approx 2.2$ and $2\pi/\log(83)\approx1.4 $.

\item Theorem \ref{T:algebra} below shows that in the presence of a Landau-Siegel zero, we may replace $\zeta_D(s)$ by
\[
\zeta_{LS}(s)=\frac{\zeta(s)L(2s,1_D)}{\zeta(s+1)^2L(s,\chi_{-D})}
\]
and see the \lq resonance\rq\  of the Landau-Siegel zero.  The corresponding formula for the logarithmic derivative shows the contribution of a zero of $L(s,\chi_{-D})$ is \emph{negative}.

 Also shown in Figure \ref{Fig:1} (in blue) is the graph of the real part of $\zeta_{LS}^\prime/\zeta_{LS}(1+2i t)$.   Of course we do not have a Landau-Siegel zero, but this discriminant is notable for having a low-lying zero (relative to the size of the discriminant), at $\rho=1/2+i\, 0.077967\ldots$\ \ 
The points marked in green on the horizontal axis correspond to the zeros of $L(1/2+2it,\chi_{-1411})$.  

\item In comparison the \lq remainder\rq\ term $\text{Rem}(it)$ is independent of $-D$ and should be small in comparison; see the remark on page \pageref{R:rem} and Figure \ref{Fig:3} for a  graph.
\end{enumerate}
\end{remarks}

\begin{figure}
\begin{center}
\includegraphics[scale=1, viewport=0 0 400 225,clip]{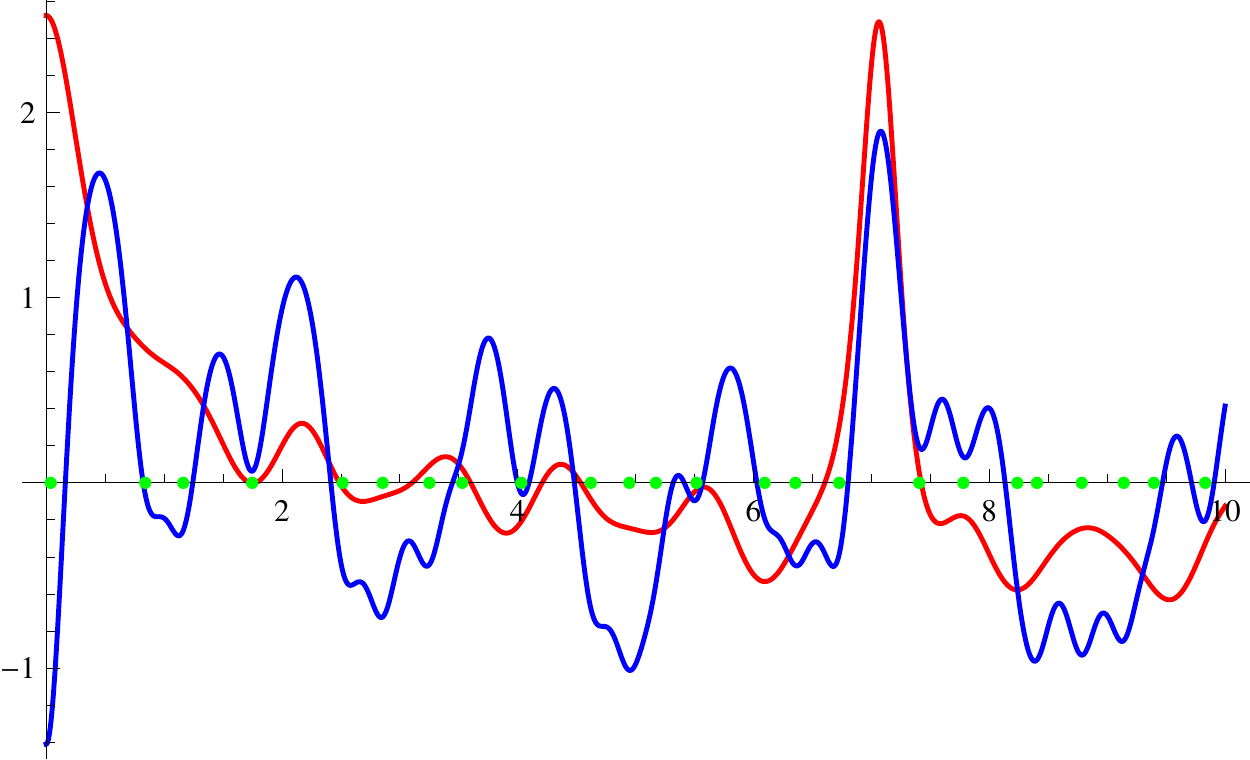}
\caption{The real parts of $\zeta_D^\prime/\zeta_D(1+2it)$ (red), and $\zeta_{LS}^\prime/\zeta_{LS}(1+2i t)$ (blue) for $-D=-1411$ and $0\le t\le 10$.}\label{Fig:1}
\end{center}
\end{figure}

In the following theorem we have the beginnings of an asymptotic expansion of the right side of (\ref{Eq:5}) in powers of the mean gap between the zeros.  Still to be accomplished in Theorems \ref{T:algebra} and \ref{T:analysis} below is a useful estimate of $S_{{\rm odd}}(\psi)$.
\begin{theorem}\label{T:Lot}
\begin{multline}\label{Eq:5p}  
\frac1{X^\ast}\sum_{f\in \mathcal{F}(X)} \sum_{\gamma}
g\left(\gamma \cdot \frac{\log (\sqrt{D}X)}{2\pi}\right)  =
2 \hg(0) -g(0) +\\
\left(2\sum_{q|D}\frac{\log(q)}{q-1}+4\frac{\zeta^\prime(2)}{\zeta(2)}-\log(4\pi^2 e^2)+2C\right)\frac{\hg(0)}{\log(\sqrt{D}X)} \\
S_{{\rm odd}}(\psi)+O\left(\frac{1}{\log(\sqrt{D}X)^2}\right).
\end{multline} 
\end{theorem}
\begin{proof}
This follows from Lemma \ref{L:gammaest} in \S \ref{S:2},  and  Lemmas \ref{L:zetaest} and \ref{L:Aest} in \S \ref{S:even}.
\end{proof}
\section{The exceptional discriminant}\label{S:0}

We saw above that the contribution to $S_{{\rm odd}}(\psi)$ of those primes which are inert is $0$.
Suppose now that $-D<0$ is a fundamental discriminant which is exceptional, that is, there is a Landau-Siegel zero $\beta=1-\delta$ of $L(s,\chi_{-D})$.  Then the inert primes will dominate and the contribution of the split primes to $S_{{\rm odd}}(\psi)$ will be  small as long as 
\[
\delta\cdot  (\sqrt{D}X)^{\sigma /2} \ll 1
\]
 is small, where the support of $\hg \subset (-\sigma,\sigma)$.  (See \S \ref{S:Appendix}).  
What we see then in (\ref{Eq:5}) and (\ref{Eq:5p}) is the \lq main term\rq\ \ $2\hg(0)-g(0)$, and some explicit $O(1/\log(\sqrt{D}X))$ corrections. 
The question is, to paraphrase Iwaniec \cite{I}
\begin{quote}\emph{
\lq\lq\  ....  for what reason
can the zeros $\rho=1/2+i\gamma$ of $L(s,\psi\otimes\chi_f)$ be so regularly distributed to generate the
functional $g\to 2\hg(0)-g(0)$?\rq\rq}
\end{quote}

Our goal is to choose interesting values of $X$ and $\sigma$ in terms of $D$ and $\delta$, in such a way as to get good estimates on $S_{{\rm odd}}(\psi)$.  The size of $\sigma$, and in particular whether $\sigma>1$ is feasible, is of interest because of this classical fact about the functional $g\to 2\hg(0)-g(0)$:
\begin{nonumlemma}
If ${\rm supp}(\hg) \subset (-\sigma, \sigma)$, then
\[
2\hg(0)-g(0) = 2\intii g(y) \left(1-\frac{\sin(2\pi \sigma y)}{2\pi y}\right)\, dy.
\]
\end{nonumlemma}
\begin{proof}
\begin{align*}
g(0) =& \intii \hg(y)\, dy = \int_{-\sigma}^\sigma \hg(y)\, dy = \intii \chi_{[-\sigma,\sigma]}(y)\hg(y) \, dy\\
\intertext{By the Plancherel Formula this is}
=& \intii \hat{\chi}_{[-\sigma,\sigma]}(y) g(y)\, dy = \intii \frac{\sin(2\pi \sigma y)}{\pi y} g(y)\, dy.
\end{align*}
\end{proof}
\noindent
The symplectic random matrix model is conjectured to model the distribution of the zeros of quadratic Dirichlet $L$-functions.  This would give the $1$-level density  as
\[
\intii g(y) \left(1-\frac{\sin(2\pi  y)}{2\pi y}\right)\, dy
\]
for \emph{all} $\sigma$.  Under the assumption of GRH, this is a theorem (up to $o(1)$ error) as long as $\sigma<2$; see \cite[Corollary 2]{OS2}.  Thus having $S_{{\rm odd}}(\psi)$ small for $\sigma>1$ tends to repel the low lying $\gamma$ away from $0$, and closer to a periodic spacing than in the symplectic random matrix model
\footnote{Be careful not to assume that $1-\sin(2\pi \sigma y)/2\pi y$ models the histogram of the zeros.  Any given bin in such a histogram may be well approximated by a Schwartz test function $g$ with very small support.  By the Uncertainty Principal, the support of the corresponding $\hg$ will be very large, exactly the opposite of our hypothesis.}.


\subsection*{Algebra}
By genus theory, we have an exact sequence for the class group $\mathcal C(-D)$:
\[
 \mathcal P(-D)\overset{\text{def.}}=\mathcal C(-D)^2 \hookrightarrow \mathcal C(-D) \twoheadrightarrow \mathcal C(-D)/ \mathcal C(-D)^2 \simeq (\mathbb Z/2)^{\omega(D)-1}.
\]
Let $e$ denote the exponent of the principal genus $\mathcal P(-D)$.
\begin{theorem}\label{T:algebra} Assume the principal genus has odd order.  (In this case the exact sequence above splits.)
Rather than use a single genus character $\psi$, we instead average  (\ref{Eq:5}) or (\ref{Eq:5p}) over all such.  For any exponent $e$, and any $\sigma<2/e$, we take $X$ no larger than 
\[
X< \frac{D^{1/(\sigma e)-1/2}}{4}.
\] 
Then  (\ref{Eq:5}) and (\ref{Eq:5p}), averaged over $\psi$, hold  with $\sum_\psi S_{{\rm odd}}(\psi)=0$.  Furthermore, the term $\zeta^\prime_D/\zeta_D$ on the right side of (\ref{Eq:5}) may be replaced by $\zeta_{LS}^\prime/\zeta_{LS}$, where
\[
\zeta_{LS}(s)=\frac{\zeta(s)L(2s,1_D)}{\zeta(s+1)^2L(s,\chi_{-D})}.
\]
\end{theorem}

\begin{proof}
The dependence on $\psi$ on the right side of (\ref{Eq:5}) is only in the term $S_{{\rm odd}}(\psi)$.  By orthogonality all primes are annihilated except those represented by a form in the principal genus, i.e. those $p$ congruent to a square modulo $D$.   

If $p$ does not divide $D$ and $p<(D/4)^{1/e}$, then $p$ is not represented by any form $Q(x,y)$ in the principal genus, because otherwise the principal form $Q^e$ would represents $p^e$.  But
\[
p^e = x^2+\frac{D}{4}y^2\quad\text{ or }\quad x^2+xy+\frac{D+1}{4}y^2
\]
with $y>0$ shows $p^e>D/4$.     The only primes which contribute to the Explicit Formula are those for which $\log(p)/\log(\sqrt{D}X)<\sigma$.  
Then with $X$ as above,
\[
p<(\sqrt{D}X)^\sigma \Rightarrow p<(D/4)^{1/e}
\]
and there are no split primes at all in the principal genus.  

The primes $q$ dividing $D$, on the other hand, are classically known to be represented by the ambiguous classes (those of order two) which are not in the principal genus by hypothesis.

We observe that the Euler products for
\[
\zeta(s)L(s,1_D)\quad\text{and}\quad\frac{\zeta(s)L(2s,1_D)}{L(s,\chi_{-D})}
\]
agree for all primes such that $\chi_{-D}(p)=-1$ or $0$, and thus for all primes $p<(\sqrt{D}X)^\sigma$.  The calculations for Lemma \ref{L:5} go through unchanged.  
\end{proof}

\begin{remarks}
\begin{enumerate}
\item The most interesting case is of course $e=1$ (which is still an open problem.)\ \   Then we can then take support of $\hg \subset (-\sigma,\sigma)$ for any $\sigma<2$.  As long as $X<D^{1/\sigma-1/2}/4$, we have the above conclusions.
\item For $e>1$, we have a strong restriction on $\sigma$, but without assuming GRH and with very small error, $O(X^{-1/2})$.
\item This result is actually unconditional as it does not explicitly refer to the Landau-Siegel zero.  Nonetheless Weinberger \cite[Theorem 4]{W}, and Boyd and Kisilevsky \cite[Theorem 4]{BK2} show that under GRH, $e\gg\log(D)/\log\log(D)$.
\item  The second part of the theorem does not require that $e$ be odd, as the Euler products for $\zeta_D(s)$ and $\zeta_{LS}(s)$ agree for all $q|D$ as well as for all $p$ with $\chi_{-D}(p)=-1$.  The first part may be generalized to the case of even $e$ as well: for the (relatively rare) small primes $q|D$, we may use the methods of Lemmas \ref{L:charsum2odd} and \ref{L:charsum2} to get a $O(X^{-\epsilon^2})$ estimate for their contribution to $S_{{\rm odd}}(\psi)$.  For the (extremely rare) large primes $q|D$, we take instead the trivial bound on $\sum_{f\in\mathcal F(X)}\chi_f(q) $.
\end{enumerate}
\end{remarks}

\subsection*{Analysis}

To get good bounds via the Burgess estimate on the character sum
\[
\sum_{f
\in \mathcal{F}(X)}\chi_f(p)
\]
appearing in  $S_{{\rm odd}}(\psi)$, we see in Lemma \ref{L:charsum2} below that $X$ must not be too small relative to the the prime $p$; we will need
\begin{equation}\label{Eq:twozero}
D^{\sigma/2}<X^{4-\sigma-16\epsilon}\quad\text{or}\quad\frac{\sigma}{8-2\sigma-32\epsilon}<\log_D(X),
\end{equation}
so
\begin{equation}\label{Eq:twoone}
\frac{\sigma}{8-2\sigma}<\log_D(X)
\end{equation}
is a clean necessary (but not sufficient!)\! condition.
On the other hand, to estimate the rest of $S_{{\rm odd}}(\psi)$,
\begin{equation}\label{Eq:23}
\sum_{l=0}^\infty \sum_p \frac{\lambda(p)\psi(Q)\log
p}{p^{(2l+1)/2}\log (\sqrt{D}X)}
 \hg\left(\frac{\log p^{2l+1}}{\log (\sqrt{D}X)}\right),
\end{equation}
we compare  in \S \ref{S:Appendix} to the analogous term in the Explicit Formula for $\zeta(s)L(s,\chi_{-D})$.  There we find that we need
\[
(\sqrt{D}X)^{\sigma /2} \cdot \delta\ll1.
\]
A theorem due to Page \cite{Page} tells us that
\[
\delta\gg \frac{1}{\sqrt{D}\log(D)^2}.
\]
In fact, the Goldfeld-Gross-Zagier lower bound on the class number, and known asymptotics \cite{Mallik, PintzII} for $\delta$ in terms of $L(1,\chi_{-D})$, tell us that
\[
\delta\gg \frac{\log(D)}{\sqrt{D}}.
\]
So
\begin{equation}\label{Eq:twotwo}
(\sqrt{D}X)^{\sigma/2}<\sqrt{D}\quad\text{or}\quad \log_D(X)<1/\sigma-1/2.
\end{equation}
is a clean necessary (but not sufficient!) condition for (\ref{Eq:23}) to be small.  (Thus necessarily $\sigma<2$).  Combining (\ref{Eq:twoone}) and (\ref{Eq:twotwo}) we have
\begin{equation}\label{Eq:twothree}
\frac{\sigma}{8-2\sigma}<\log_D(X)<1/\sigma-1/2
\end{equation}
These inequalities coalesce at $\sigma=4/3$ and $X=D^{1/4}$, which motivates what follows.
\begin{theorem}\label{T:analysis}
For any $\sigma<4/3$, let $X<D^{1/4}$ such that (\ref{Eq:twothree}) holds, and choose $\epsilon$ so that (\ref{Eq:twozero}) holds.  Suppose that $g$ and $\hg$ are non-negative.  Under Hypothesis H, we may omit the term $S_{{\rm odd}}(\psi)$ on the right side of (\ref{Eq:5}) and (\ref{Eq:5p}), at the expense of multiplying the right side by 
\begin{equation}\label{Eq:bigofirst}
1+O_\epsilon\left(\tau(D)\log(\omega(D))X^{-\epsilon^2}\right),
\end{equation}
and including an additional error of
\begin{equation}\label{Eq:bigosecond}
O\left(\max \hg \cdot(\sqrt{D}X)^{\sigma /2} \cdot \delta\cdot X^{-\epsilon^2}\right).
\end{equation}
\end{theorem}
\begin{remark}
As in Theorem \ref{T:algebra}, we could in Theorem \ref{T:analysis} replace the term $\zeta_D^\prime/\zeta_D$ by $\zeta_{LS}^\prime/\zeta_{LS}$, at the cost of introducing an additional error term to account for the (relatively rare) primes $p<(\sqrt{D}X)^\sigma$ for which $\chi_{-D}(p)=1$.
\end{remark}
\begin{proof}
We can use  the character sum estimate (\ref{Eq:PVbound}) with (\ref{Eq:mertens}) to bound $S_{{\rm odd}}(\psi)$ as
\[
 \ll_\epsilon 
\frac{\tau(D)\log(\omega(D))}{X^{\epsilon^2}}  \sum_{l=0}^\infty \sum_p \frac{\lambda(p)\log
p}{p^{(2l+1)/2}\log (\sqrt{D}X)}
 \hg\left(\frac{\log p^{2l+1}}{\log (\sqrt{D}X)}\right).
\]
The double sum is the absolute value of  (\ref{Eq:23}), which is less than the left side of (\ref{Eq:17}).  But on the right side of (\ref{Eq:17}) we see the same terms as in the right side of  (\ref{Eq:5}) which requires including the factor (\ref{Eq:bigofirst}), as well as the contribution of the pole and Landau-Siegel zero leading to the error term (\ref{Eq:bigosecond}).
\end{proof}

\section{The conductors term}\label{S:1}

\begin{nonumlemma}
Let $f$ denote a
fundamental discriminant with $|f|<X$. Then 
\begin{equation}\label{Eq:charsum1}
\sum_{|f| \le X} 1   = \frac{6}{\pi^2}X + O(X^{1/2}) 
\end{equation}
and for fixed $p$
\begin{equation}\label{Eq:charsum2}
 \sum_{\substack{|f| \le X \\ p|f}} 1 \ = \ \frac{6X}{\pi^2(p+1)} +
O(X/p)^{1/2}.
\end{equation}
\end{nonumlemma}
\begin{proof}
This can be shown following the very straightforward approach given in \cite[Appendix B]{Miller}, modified only with the improvement $1/p^{1/2}$ in the $O(\ )$.
\end{proof}

In (\ref{Eq:charsum1}), replace $X$ with $2X$ and subtract to deduce that
\[
\sum_{X\le |f| <2X}1 = \frac{6}{\pi^2}X+O(X^{1/2}).
\]
Via an inclusion-exclusion argument, we then deduce from (\ref{Eq:charsum1}) that
\[
X^\ast=\sharp \mathcal F(X) = \frac{6}{\pi^2}\prod_{q|D}\left(1-\frac{1}{q+1}\right)X+O(X^{1/2}).
\]
We will later have need of an estimate for $1/X^*$ in terms of $D$, i.e. the size of $\prod_{q|D}(1-1/(q+1))^{-1}$.  With
$
\omega(D)=\sharp\left\{q|D\right\}
$,
we can estimate
\[
\prod_{q|D}(1-1/(q+1))^{-1}<\prod_{i=1}^{\omega(D)}(1-1/p_i)^{-1}\ll \log(\omega(D))
\]
by Mertens' Formula.  So
\begin{equation}\label{Eq:mertens}
\frac{1}{X^\ast}\ll\frac{\log(\omega(D))}{X}.
\end{equation}

We are now ready to analyze the first term on the right side of (\ref{Eq:4})
\[
\frac1{X^*} \sum_{f \in
\mathcal{F}(X)}  \frac{\log\left(Df^2/\pi^2\right)}{\log(\sqrt{D} X)}\intii g(\tau)d\tau.
\]
Of course, the integral
$\intii g(\tau)d\tau$ is just $\hg(0)$.
Certainly we can re-write
\[
\log\left(\frac{Df^2}{\pi^2}\right)=\log\left(\frac{D}{\pi^2}\right)+2\log|f|;
\]
the first term, when summed over $f$, cancels the $1/X^*$, contributing a term
$\log(D/\pi^2)\hg(0)/\log(\sqrt{D} X)$.
\begin{lemma}
\[
\frac1{X^*\log(\sqrt{D} X)} \sum_{f \in \mathcal F(X)}  2\log|f| =\frac{\log(16X^2/e^2)}{\log(\sqrt{D}X)}+O\left(\frac{\log(\omega(D))}{X^{1/2}\log(\sqrt{D}X)}\right).
\]
\end{lemma}
\begin{proof} This is just partial summation as in \cite[Thm 4.2]{Apostol}.  We let
\[
A(x)=\sum_{\substack{X\le |n|<x \\ n \text{ fund., }(n,D)=1}}1,\qquad h(x) = 2\log|x|
\]
so
\begin{align*}
\sum_{f \in \mathcal F(X)}  2\log|f| 
=& \sum_{n<2X} h(n) dA(n)\\
=&A(2X)h(2X)-\int_X^{2X} A(t)h^\prime(t)\, dt\\
=&2X^*\log(2X)-\\
&\quad\int_X^{2X} \left(\frac{6}{\pi^2}\prod_{q|D}\left(1-\frac{1}{q+1}\right)(t-X)+O(t^{1/2})\right)2t^{-1}\, dt\\
=&2X^*\log(2X)-2X^\ast+2X^*\log(2)+O(X^{1/2}).
\end{align*}
Dividing by $X^*\log(\sqrt{D}X)$ gives the lemma.

\end{proof}

Combining with the easy previous term, we see that the first line on the right side of (\ref{Eq:4}) is
\begin{multline*}
\frac{\log(DX^2\cdot 16/\pi^2e^2)}{\log(\sqrt{D} X)}\hg(0)+O\left(\frac{\hg(0)\log(\omega(D))}{X^{1/2}\log(\sqrt{D}X)}\right)=\\
\left(2+\frac{\log(16/\pi^2e^2)}{\log(\sqrt{D}X)}\right) \hg(0) +O\left(\frac{\hg(0)\log(\omega(D))}{X^{1/2}\log(\sqrt{D}X)}\right).
\end{multline*}

\section{The Gamma terms}\label{S:2}

In this section we show that the second line on the right side of (\ref{Eq:5}), coming from the Gamma factors, simplifies, at the expense of rewriting in term of $\hg$ instead of $g$.
In
\begin{multline}\label{Eq:gammamess}
\frac{1}{\log(\sqrt{D}X)}\intii g(\tau)\tre\Bigg[
\frac{\Gamma'}{\Gamma}\left(\frac14+\frac{i\pi
\tau}{\log(\sqrt{D}X)}\right)+\\
\frac{\Gamma'}{\Gamma}\left(\frac34+\frac{i\pi
\tau}{\log(\sqrt{D}X)}\right)
\Bigg]
d\tau
\end{multline}
we change variables $t=\tau/\log(\sqrt{D}X)$.  We will next apply \cite[Lemma 12.14]{MV} which serves as a sort of substitute for the Plancherel Theorem in this context.
Apply the lemma twice, with $b=\pi$ and $a=1/4$ (resp. $3/4$).  Writing $\widehat{J}(t)=g(\log(\sqrt{D}X) t)$, the usual Fourier identities give $J(y)=\hg(y/\log(\sqrt{D}X))/\log(\sqrt{D}X)$.  The lemma \emph{op. cit.} says that (\ref{Eq:gammamess}) is
\begin{multline*}
\left(\frac{\Gamma'}{\Gamma}(1/4)+\frac{\Gamma'}{\Gamma}(3/4)\right)\frac{\hg(0)}{\log(\sqrt{D}X)}+\\
\frac{1}{\log(\sqrt{D}X)}\int_0^\infty\frac{\hg(0)-\hg(y/\log(\sqrt{D}X))}{\sinh(y/2)}dy,
\end{multline*}
after some fiddling with the exponentials and recalling that $\hg$ is even.  
For future reference we note that
\[
\frac{\Gamma'}{\Gamma}(1/4)+\frac{\Gamma'}{\Gamma}(3/4)=-2C-\log(64)
\]
where $C$ is the Euler constant.
\begin{lemma}\label{L:gammaest}  
\[
\frac{1}{\log(\sqrt{D}X)}\int_0^\infty\frac{\hg(0)-\hg(y/\log(\sqrt{D}X))}{\sinh(y/2)}dy\ll\frac{1}{\log(\sqrt{D}X)^2}.
\]
\begin{proof}
Since $g$ is Schwartz, $\hg(y)=\hg(0)+O(y)$.
The lemma follows immediately from the fact that
\[
\int_0^\infty\frac{y}{\sinh(y/2)}dy=\pi^2.
\]
\end{proof}
\end{lemma}

\section{The sum over even powers of primes}\label{S:even}
Let
\[
\zeta(s)L(s,1_D)=\zeta(s)^2\prod_{q|D}(1-q^{-s})
\]
and
\[
\Lambda_D(n)=\begin{cases}2\log(p)&\text{ if }n=p^k, (p,D)=1\\
\log(q)&\text{ if }  n=q^k, q|D\\
0&\text{ otherwise,}
\end{cases}
\]
so
\begin{equation}\label{Eq:LambdaD}
\frac{\zeta^\prime(s)}{\zeta(s)}+\frac{L^\prime(s,1_D)}{L(s,1_D)}=-\sum_n\Lambda_D(n)n^{-s}.
\end{equation}

\begin{lemma}\label{L:5}
Suppose $\hg$ has finite support. Then 
\begin{multline}\label{Eq:6}
S_{{\rm even};1}  = -g(0) + \frac{2}{\log(\sqrt{D}X)}\times \\
\intii g(\tau)\tre\left[
\frac{\zeta^\prime}{\zeta}\left(1+\frac{4\pi i \tau}{\log(\sqrt{D}X)}\right)+
\frac{L^\prime}{L}\left(1+\frac{4\pi i \tau}{\log(\sqrt{D}X)},1_D\right) \right]d\tau.
\end{multline}
\end{lemma}
\begin{remark} I'm not sure of the origin of this key lemma, but I expect it must be implicit in \cite{OS}.  The proof given here follows the elegant treatment of Miller in \cite{Miller} in all essentials; it is included here merely for completeness.  
\end{remark}
\begin{proof}
We have 
\[
S_{{\rm even};1}  =  \frac{-2}{\log(\sqrt{D}X)}\sum_{n=1}^\infty \frac{\Lambda_D(n)}{n} \ \hg\left(2\frac{\log
n}{\log (\sqrt{D}X)}\right). 
\]
For any $\epsilon > 0$ define\footnote{For this Lemma we use $\delta$ and $\epsilon$ as generic small parameters, not in the global sense they have in the rest of the paper.}
\[
 I(\epsilon)  =
 \frac{1}{2\pi i} \int_{\tre(z)=1+\epsilon} g\left(\frac{(2z-2)\log A}{4\pi i}\right)
\sum_{n=1}^\infty \frac{\Lambda_D(n)}{n^z}\ dz; 
\]
we will later
take $A = D^{1/4}X^{1/2}$.

Miller re-writes $I(\epsilon)$ by shifting contours while avoiding poles. For $\delta > 0$ consider the contour
made up of three pieces: $(1-i\infty,1-i\delta]$, $C_\delta$, and
$[1+i\delta,1+i\infty)$, where 
\[
C_\delta = \{z = 1+ \delta
e^{i\theta}, \theta \in [-\pi/2,\pi/2]\}
\]
is the semi-circle going
counter-clockwise from $1-i\delta$ to $1+i\delta$. By Cauchy's
residue theorem, the contour in $I(\epsilon)$ can be shifted from $\tre(z) =
1+\epsilon$ to the three curves above.  Recalling (\ref{Eq:LambdaD}), $ I(\epsilon)$ is equal to
\[
\frac{1}{2\pi i}\int_{1-i\infty}^{1-i\delta} + \int_{C_\delta} +
\int_{1+i\delta}^{1+i\infty} g\left(\frac{(2z-2)\log A}{4\pi
i}\right) \left(-\frac{\zeta^\prime(z)}{\zeta(z)}-\frac{L^\prime(z,1_D)}{L(z,1_D)}\right)\ dz.
\]
The limit as $\delta \to 0$ of the integral
over $C_\delta$ is evaluated as follows.   One can write
\begin{multline*}
g\left(\frac{(2z-2)\log A}{4\pi
i}\right) \left(-\frac{\zeta^\prime(z)}{\zeta(z)}-\frac{L^\prime(z,1_D)}{L(z,1_D)}\right) = \\
g(0)\cdot \frac{2}{z-1}+\text{holomorphic}.
\end{multline*}
(Recall that $\zeta(s)^2$ has a double pole at $s=1$, while the factors from $q|D$ are holomorphic.)\ \ 
The contribution of the pole is $g(0)$, independent of what $\delta$ is, while the holomorphic piece tends to $0$ with $\delta$ by usual bound on integrands and path lengths.
Now take the limit as $\delta \to 0$ in what remains: 
\begin{multline*}
g(0)- I(\epsilon) =\\
\lim_{\delta \to 0} \frac1{2\pi}
\int_{-\infty}^{-\delta} + \int_{\delta}^\infty
g\left(\frac{y\log A}{2\pi}\right)
\left(\frac{\zeta^\prime(1+iy)}{\zeta(1+iy)}+\frac{L^\prime(1+iy,1_D)}{L(1+iy,1_D)}\right)dy.
\end{multline*}
Miller claims the limit of the integral above is well-defined.
For large $y$ this follows from the decay of $g$.  For small
$y$ it follows from the fact that 
\[
\frac{\zeta^\prime(1+iy)}{\zeta(1+iy)}+\frac{L^\prime(1+iy,1_D)}{L(1+iy,1_D)} = \frac{-2}{iy}+O(1).
\]
The contribution of the pole is an odd function of $y$, so orthogonal to $g$ which is even. 
The imaginary part of $\zeta^\prime/\zeta(1+iy)+L^\prime/L(1+iy,1_D)$ is also an odd function of $y$, only the real part survives.  Take $A=D^{1/4}X^{1/2}$,
and change variables to $\tau = y\log (A)/2\pi$ which is $y\log(\sqrt{D}X)/4\pi$. Thus 
\begin{multline*}
 I(\epsilon) = g(0) - \frac{2}{\log(\sqrt{D}X)}\times\\
 \intii
g(\tau)\tre\left[ \frac{\zeta^\prime}{\zeta}\left(1+\frac{4\pi i \tau}{\log(\sqrt{D}X)}\right)
+\frac{L^\prime}{L}\left(1+\frac{4\pi i \tau}{\log(\sqrt{D}X)},1_D\right) \right]d\tau.
\end{multline*}
Observe that his beautiful formula is independent of $\epsilon$!

On the other hand, in the original definition of $I(\epsilon)$ (before the contour was moved), write $z = 1+\epsilon+iy$.  We will use
Fourier analysis to write
$g(x+iy)$ in terms of the transform of $\hg(u)$.  
Normalize the Fourier transform so that
\begin{gather*}
\hat g(u) = \intii g(x) e^{-2\pi i x u}dx\\
g(x) =  \intii \hat g(u) e^{2\pi i xu} du \\
g(x+iy) = \intii \hat g(u) e^{2\pi i(x+iy)u} du.
\end{gather*}

He has that 
\[
I(\epsilon)= \sum_{n=1}^\infty \frac{\Lambda_D(n)}{n^{1+\epsilon}} \frac{1}{2\pi i} \intii
g\left(\frac{y\log A}{2\pi}-\frac{i\epsilon\log
A}{2\pi}\right)\exp(-iy\log n) idy 
\]
which is equal to
\begin{multline*} 
\sum_{n=1}^\infty
\frac{\Lambda_D(n)}{n^{1+\epsilon}} \frac1{2\pi} \intii \exp(-iy\log n) \cdot\\
\intii 
\hg(u)\exp(\epsilon u \log A) \exp(2\pi i \frac{y\log A}{2\pi
}u) du\, dy. 
\end{multline*}
Let $h_\epsilon(u) = \hg(u)
\exp(\epsilon u \log A)$. Note that $\widehat{\widehat{h_\epsilon}}(w) = h_\epsilon(-w)$.
Thus 
{\allowdisplaybreaks
\begin{align*} 
I(\epsilon)  =  & \sum_{n=1}^\infty
\frac{\Lambda_D(n)}{n^{1+\epsilon}}  \frac1{2\pi} \intii
\widehat{h_\epsilon}\left(-\frac{y\log A}{2\pi}\right) \exp(-iy\log n) dy\\ 
=& \sum_{n=1}^\infty \frac{\Lambda_D(n)}{n^{1+\epsilon}}
\intii \widehat{h_\epsilon}(y) \exp(-2\pi i y (-\frac{\log n}{\log A}))\
\frac{ dy}{\log A} \\ 
=& \frac1{\log A}\sum_{n=1}^\infty
\frac{\Lambda_D(n)}{n^{1+\epsilon}} \
\widehat{\widehat{h_\epsilon}}\left(-\frac{\log n}{\log A}\right)\\
=&  \frac1{\log A}\sum_{n=1}^\infty \frac{\Lambda_D(n)}{n^{1+\epsilon}} \
\hg\left(\frac{\log n}{\log A}\right) \exp(\epsilon \log n) \\
=& \frac1{\log A}\sum_{n=1}^\infty \frac{\Lambda_D(n)}{n}\
\hg\left(\frac{\log n}{\log A}\right). 
\end{align*}
}
By again taking
$A = D^{1/4}X^{1/2}$ we find 
\[
I(\epsilon)  = \frac{2}{\log(\sqrt{D}X)}\sum_{n=1}^\infty \frac{\Lambda_D(n)}{n}  \hg\left(2\frac{\log
n}{\log(\sqrt{D}X)}\right) = -S_{{\rm even};1}.
\]

\end{proof}

As with the Gamma terms in the previous section, we can estimate the integral:
\begin{lemma}\label{L:zetaest} 
\begin{multline*}
\frac{2}{\log(\sqrt{D}X)}\intii g(\tau)\tre\left[
\frac{\zeta^\prime}{\zeta}\left(1+\frac{4\pi i \tau}{\log(\sqrt{D}X)}\right)
+\frac{L^\prime}{L}\left(1+\frac{4\pi i \tau}{\log(\sqrt{D}X)},1_D\right)
\right] d\tau\\
=2\left(2C+\sum_{q|D} \frac{\log(q)}{q-1}\right)\cdot\frac{\hg(0)}{\log(\sqrt{D}X)}+O\left( \frac{1}{ \log(\sqrt{D}X)^2 }\right).
\end{multline*}
\end{lemma}
\begin{proof}
We change variables
\[
t=\frac{\tau}{\log(\sqrt{D}X)}\qquad dt = \frac{d\tau}{\log(\sqrt{D}X)}\qquad g(\tau) = g(\log(\sqrt{D}X) t).
\]
We have lost the \lq obvious\rq\  $1/\log(\sqrt{D}X)$ term, but since $g$ is Schwartz,
\[
g(\log(\sqrt{D}X) t) \ll \frac{1}{\log(\sqrt{D}X)^2t^2 }.
\]
From \cite[(3.11.9)]{Tit} we have
\[
\frac{\zeta^\prime}{\zeta}\left(1+i t\right)\ll \log(t),
\]
while the terms arising from the $q|D$ are periodic and hence $O(1)$.
So to treat the integral on $(-\infty,-1]\cup[1,\infty)$, we bound the integral by
\[
\ll \frac{1}{\log(\sqrt{D}X)^2}\int_1^\infty \frac{\log (t)}{t^2}\, dt= \frac{1}{\log(\sqrt{D}X)^2}.
\]
Meanwhile, 
\begin{gather*}
\tre\left[2\frac{\zeta^\prime}{\zeta}(1+4\pi i t)\right] = 2C+O(t^2),\\
\tre\left[\frac{\log(q)}{q^{1+i t}-1}\right] = \frac{\log(q)}{q-1}+O(t^2)
\end{gather*}
(where $C$ is the Euler constant).   We treat the constant terms and the quadratic error separately, computing that
\begin{align*}
2\int_{-1}^1 g(\log(\sqrt{D}X) t)\, dt =&\frac{2}{\log(\sqrt{D}X)}\int_{-\log(\sqrt{D}X)}^{\log(\sqrt{D}X)}g(y)\,dy\\
=&\frac{2\hg(0)}{\log(\sqrt{D}X)}-
\frac{4}{\log(\sqrt{D}X)}\int_{\log(\sqrt{D}X)}^{\infty}g(y)\,dy\\
=&\frac{2\hg(0)}{\log(\sqrt{D}X)}+O\left(\frac{1}{\log(\sqrt{D}X)^2}\right),
\end{align*}
as $g(y)\ll 1/y^2$.
And
\begin{align*}
\int_{-1}^1 g(\log(\sqrt{D}X) t)\cdot t^2\, dt=&\frac{1}{\log(\sqrt{D}X)^3}\int_{-\log(\sqrt{D}X)}^{\log(\sqrt{D}X)}g(y)\cdot y^2\,dy\\
 \ll& \frac{1}{ \log(\sqrt{D}X)^2 },
\end{align*}
since the integrand is $O(1)$.
\end{proof}

\begin{lemma}  Let
\[ 
A'(r) = \sum_p \frac{\log p}{(p+1)(p^{1+2r}-1)}.
\]
Then
\begin{multline}
S_{{\rm even};2} =  \frac{4}{\log(\sqrt{D}X)}\intii
g(\tau) \tre \left[A'\left(\frac{2\pi i \tau}{\log
(\sqrt{D}X)}\right)\right]d\tau \\+ O\left(\frac{\max\hg\cdot \log(\omega(D))}{X^{1/2}\log(\sqrt{D}X)}\right).
\end{multline}
\end{lemma}
\begin{proof}
The proof is again that of \cite{Miller}, simplified slightly as the error bounds in (\ref{Eq:charsum2}) track the dependence on $p$.
Observe that in (\ref{Eq:seven2}) we have $p|f\le 2X$, so certainly $p\le 2X$.
Thus (\ref{Eq:seven2}) is
\begin{align*}
& =  \frac4{X^\ast}
\sum_{f\in\mathcal F(X)} \sum_{\ell=1}^\infty \sum_{\substack{p \le 2X \\ p|f}}
\frac{\log p}{p^\ell \log(\sqrt{D}X)}\ \hg\left(2\frac{\log p^\ell}{\log
(\sqrt{D}X)}\right) \\ 
&=\frac4{X^\ast} \sum_{\ell=1}^\infty \sum_{p \le 2X} \frac{\log
p}{p^\ell \log (\sqrt{D}X)}\ \hg\left(2\frac{\log p^\ell}{\log (\sqrt{D}X)}\right)
\sum_{f\in\mathcal F(X),\ p|f} 1 \\
&= 4 \sum_{\ell=1}^\infty \sum_{p \le
2X} \frac{\log p}{p^\ell \log (\sqrt{D}X)}\cdot \frac{1}{p+1}\
\hg\left(2\frac{\log p^\ell}{\log (\sqrt{D}X)}\right) \\
&\qquad\qquad\qquad+
O\left(\frac{\max\hg\cdot X^{1/2}}{X^*\log(\sqrt{D}X)} \sum_{\ell=1}^\infty \sum_{ p\le 2X}
\frac{\log(p)}{\sqrt{p}p^\ell}\right). 
\end{align*}
In the $O$ term, sum the series on $\ell$ first, each is again $\ll \log(p)/p^{3/2}$, the sum of these converges.  This gives
\begin{multline*}
= 4
\sum_{\ell=1}^\infty \sum_{p \le 2X} \frac{\log p}{p^\ell \log
(\sqrt{D}X)}\cdot \frac{1}{p+1}\ \hg\left(2\frac{\log p^\ell}{\log (\sqrt{D}X)}\right) \\
 +O\left(\frac{\max\hg\cdot\log(\omega(D))}{X^{1/2}\log(\sqrt{D}X)}\right).
\end{multline*}
We see the error the Lemma allows and re-write the rest, the terms involving $\hg(2\log p^\ell/\log (\sqrt{D}X))$ by expanding the Fourier
transform.  It is equal to
\begin{multline*}    \frac{4}{\log (\sqrt{D}X)}
\sum_{\ell=1}^\infty \sum_{p \le 2X} \frac{\log p}{(p+1)p^\ell
} \times\\
\intii g(\tau) \exp(-2\pi i\tau \cdot 2\log p^\ell / \log (\sqrt{D}X))
d\tau \\
=  \frac{4}{\log (\sqrt{D}X)} \sum_{p \le 2X} \frac{\log p}{(p+1)}\intii g(\tau)
\sum_{\ell=1}^\infty p^{-\ell} \cdot p^{-4\pi i\tau  \ell/\log
(\sqrt{D}X)} d\tau 
\end{multline*}
Sum the geometric series to get
\[
=  \frac{4}{\log (\sqrt{D}X)} \sum_{p\le 2X}
\frac{\log p}{(p+1)}\intii g(\tau) \left( p^{1+4\pi i\tau/\log
(\sqrt{D}X)}-1\right)^{-1}d\tau . 
\]
We claim we can extend the $p$-sum by putting in the primes $p>2X$ at a cost of another error no worse than
$O(X^{-0.999})$.  This is because the summands are $O(\log p /
p^2)$ and $g$ is bounded. So we are claiming
\[
\sum_{p>2X}\frac{\log p}{p^2}\times \text{bounded}\ll\sum_{n>2X} \frac{\log n}{n^2}\ll X^{-0.999},
\]
by the integral test.  Since this depends on $\max g$, we should really keep track of this as well, but the $X^{-0.999}$ is so small we'll just ignore it.
The resulting $p$-sum is 
$A'(2\pi i \tau/\log (\sqrt{D}X))$.  As before, the imaginary part is orthogonal to $g$.
\end{proof}

\begin{remark}
The function $A^\prime(r)$ arises from a derivation of 1-level density via the $L$-functions Ratio Conjecture, see \cite{Miller} (where the notation is $A_D^\prime(r,r)$).
\end{remark}

\begin{lemma}\label{L:Aest} We have
\begin{multline*}
\frac{4}{\log(\sqrt{D}X)}\intii
g(\tau) \tre\left[ A'\left(\frac{2\pi i \tau}{\log
(\sqrt{D}X)}\right)\right]d\tau=\\
4\frac{\zeta^\prime(2)}{\zeta(2)}\frac{\hg(0)}{\log(\sqrt{D}X)}+O\left(\frac{1}{\log(\sqrt{D}X)^2}\right)
\end{multline*}
\end{lemma}
\begin{proof}
The proof is similar to that of Lemma \ref{L:zetaest}, but is in fact easier as the series $A^\prime(r)$ is absolutely convergent when $r$ is purely imaginary, and bounded by $A^\prime(0)$, since for $|z|=1$
\[
|pz-1|\ge p-1,
\]
and
\[
A^\prime(0)=\sum_p \frac{\log(p)}{p^2-1}=-\frac{\zeta^\prime(2)}{\zeta(2)}\approx 0.569961\ldots .
\]
\end{proof}

We break $A^\prime(r)$ into the sum of two terms:
\[
A^\prime(r)=-\frac{\zeta^\prime}{\zeta}(2+2r)+\left(A^\prime(r)+\frac{\zeta^\prime}{\zeta}(2+2r)\right),
\]
and let $\text{Rem}(r)$ denote the term in parenthesis above.  
\begin{remark}\label{R:rem}
The sum over primes in $\text{Rem}(it)$ converges no slower than $\sum_p\log(p)/p^3$.  Furthermore the term in the sum corresponding to a prime $p$ is periodic and equal to $0$ whenever $t\log(p)/2\pi$ is an integer, so we expect there is a lot of cancellation in the sum.  Figure \ref{Fig:3} shows a graph comparing $\tre[\zeta^\prime/\zeta(1+2it)]$, $\tre[-\zeta^\prime/\zeta(2+2it)]$ and $\tre[\text{Rem}(it)]$ for $0\le t\le 20$.
\end{remark}

\begin{figure}
\begin{center}
\includegraphics[scale=1, viewport=0 0 400 225,clip]{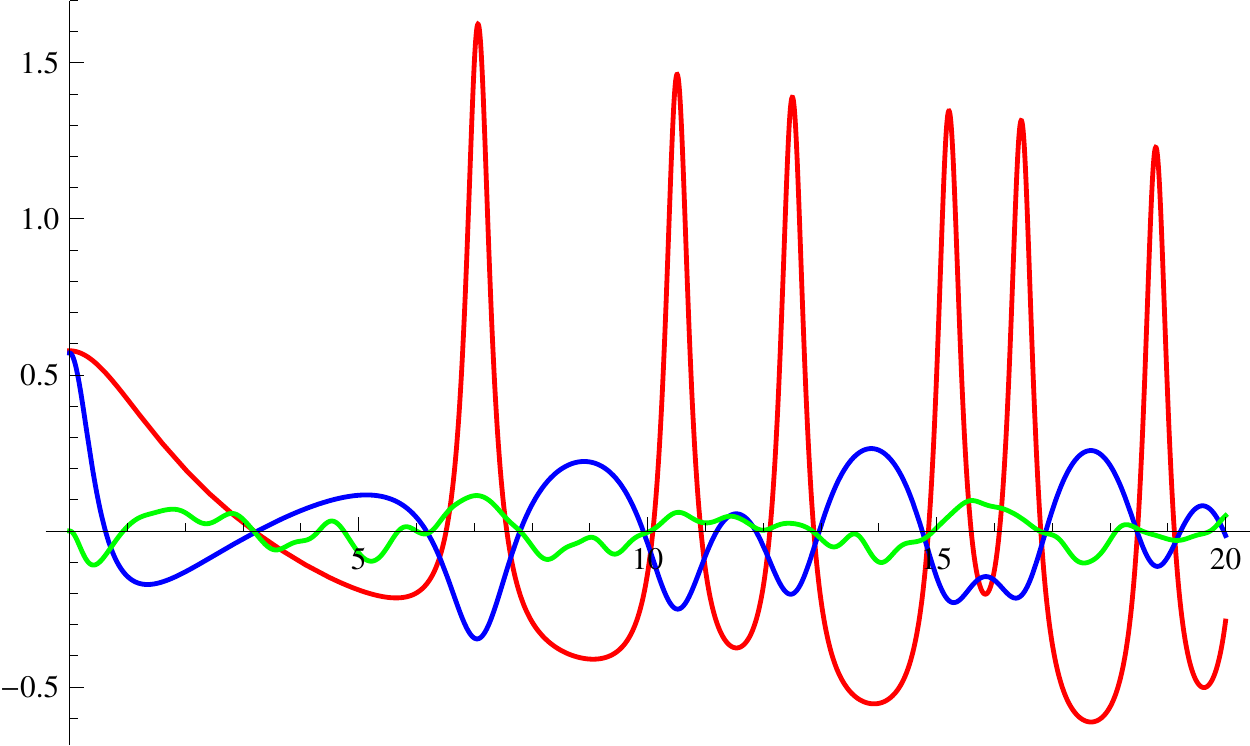}
\caption{$\tre[\zeta^\prime/\zeta(1+2it)]$ (red), $\tre[-\zeta^\prime/\zeta(2+2it)]$ (blue), and $\tre[\text{Rem}(it)]$ (green) for $0\le t\le 20$}\label{Fig:3}
\end{center}
\end{figure}

\section{The sum over odd powers of primes}\label{S:odd} 

Throughout we let $1$ the trivial character modulo $4$ and
$\chi(n)$ be the non-trivial character modulo $4$:
\[
1(n) =
\begin{cases}
1&\text{if } n\equiv 1,3 \bmod 4,\\
0&\text{if } n\equiv 0,2 \bmod 4
\end{cases}
\quad
\chi(n) =
\begin{cases}
1&\text{if } n\equiv 1 \bmod 4,\\
-1&\text{if } n\equiv 3 \bmod 4,\\
0&\text{if } n\equiv 0,2 \bmod 4
\end{cases}
\]

\begin{lemma}\label{L:charsum2odd} For an odd prime $p$,
\begin{equation}\label{Eq:PVboundodd}
\sum_{\text{odd }f\in\mathcal F(X)}\chi_f(p) \ll_\epsilon 
\tau(D)\left(X^{1-\epsilon^2}+X^{1/2}p^{1/8+\epsilon/2}  \right),
\end{equation}
where $\tau(D)$ denotes the number of divisors of $D$, and the implied constant depends only on $\epsilon$.
\end{lemma}

\begin{proof}  First note that by the triangle inequality, it suffices to prove the bound for $0<|f|\le X$ rather than $f\in\mathcal F(X)$.   
Separating the odd positive and odd negative fundamental discriminants, we have a sum of two terms (choosing $+$ or $-$ consistently below)
\[
\sum_{\substack{ 0<f\le X\\ \pm f \text{ odd fund.}\\(f,D)=1}}\chi_{\pm f}(p) 
=\frac{1}{2}\sum_{\substack{0<f\le X\\(f,D)=1\\f \text{ squarefree}}} \left(\frac{\pm f}{p}\right)\left(1(f)\pm \chi(f)\right).
\]
The possible choices of $+$, $-$, $1$ and $\chi$ leads to four terms all with the same analysis; we will consider the one involving $-$ and $\chi$.  Ignore the constant $(-1/p)/2$, and write the characteristic function of square free integers $f$ prime to $D$ as
\[
\sum_{l^2|f}\mu(l)\sum_{t|(f,D)}\mu(t).
\]
Observe that the contribution of $f$ is $0$ unless $f$ is odd, thus (since $t|D$) we may assume $t$ is square free, and writing $f=ut$ we may assume $l^2|u$, or $u=l^2k$.  Thus we have
\begin{multline*}
\sum_{\substack{0<f\le X\\(f,D)=1}}\mu^2(f) \left(\frac{f}{p}\right)\chi(f)=
\sum_{t|D}\mu(t)\left(\frac{t}{p}\right)\chi(t)\\
\times \sum_{0<l\le (X/t)^{1/2}}\mu(l)\left(\frac{l^2}{p}\right)\chi(l^2)\sum_{0<k\le X/tl^2}\left(\frac{k}{p}\right)\chi(k).
\end{multline*}
To summarize our progress, we can bound our sum with
\[
\sum_{t|D}\sum_{l<(X/t)^{1/2}}\text{ a character sum mod } 4p \text{ of length }\frac{X}{tl^2}.
\]
Since we need to consider $p$ as large as $\left(\sqrt{D}X\right)^\sigma$, and hope for $\sigma>1$, this is a short character sum.  The Burgess bound \cite[(12.57)]{IK} gives
\[
\sum_{0<k\le X/tl^2}\left(\frac{k}{p}\right)\chi(k) \ll_\epsilon \left(\frac{X}{tl^2}\right)^{1-\epsilon^2}
\]
as long as $X/tl^2>p^{1/4+\epsilon}$.  We proceed by breaking the sum on $l$ at the point where we must fall back on the trivial bound instead, giving
\[
\sum_{t|D}\:\sum_{l<(X/tp^{1/4+\epsilon})^{1/2}}\left(\frac{X}{tl^2}\right)^{1-\epsilon^2}+
\sum_{t|D}\:\sum_{(X/tp^{1/4+\epsilon})^{1/2}<l}\frac{X}{tl^2}.
\]
The first double sum above looks like
\[
X^{1-\epsilon^2}\sum_{t|D}t^{\epsilon^2-1}\sum_{l<(X/tp^{1/4+\epsilon})^{1/2}}l^{2\epsilon^2-2}\ll_\epsilon \tau(D) X^{1-\epsilon^2}
\]
since the sum on $l$ is $O(1)$.  Meanwhile, the second double sum is
\[
\ll 
\sum_{t|D}\frac{X}{t}\cdot  \left(\frac{tp^{1/4+\epsilon}}{X}\right)^{1/2}\ll\tau(D) X^{1/2}p^{1/8+\epsilon/2},
\]
where the first $\ll$ comes from comparing a sum to an integral.

\end{proof}

\begin{lemma}\label{L:charsum2} Under our hypothesis that the support of $\hg \subset (-\sigma,\sigma)$, we have $p<\left(\sqrt{D}X\right)^\sigma$.  Then as long as
\[
D^{\sigma/2}<X^{4-\sigma-16\epsilon}
\]
we have
\begin{equation}\label{Eq:PVbound}
\sum_{f\in\mathcal F(X)}\chi_f(p) \ll_\epsilon  
\tau(D)X^{1-\epsilon^2},
\end{equation}
where $\tau(D)$ denotes the number of divisors of $D$, and the implied constant depends only on $\epsilon$.
\end{lemma}

\begin{proof}
Consider first odd $p$ and odd $f$ as above.  Then
\[
X^{1/2}p^{1/8+\epsilon/2}<X^{1-\epsilon^2}
\]
as long as
\[
\left(\sqrt{D}X\right)^\sigma<X^{(1/2-\epsilon^2)/(1/8+\epsilon/2)},
\]
and
\[
\left(\sqrt{D}X\right)^\sigma<X^{4-16\epsilon}
\]
suffices,
by truncating the  expansion of $(1/2-\epsilon^2)/(1/8+\epsilon/2)$ as an alternating series.

Next, consider odd $p$ and arbitrary $f$.  The fundamental discriminants with $|f|<X$ are either odd fundamental discriminants, or of the form $-4f^\prime$ with odd $f^\prime$ and $|f^\prime|<X/4$, or of the form $\pm8f^\prime$ with $f^\prime$ odd and $|f^\prime|<X/8$.  Break up the sum into four sums accordingly, and factor a term $\chi_{-4}(p)$, $\chi_{8}(p)$, $\chi_{-8}(p)$, out of the last three.  Four applications of (\ref{Eq:PVboundodd}) and the above give the Lemma.   

Finally, consider the case $p=2$.  Now $\chi_f(2)=0$ unless $f$ is odd, in which case it depends only on $f$ modulo $8$.  We proceed much as in Lemma \ref{L:charsum2odd}, with the difference being that we use characters of the multiplicative group modulo $8$ to pick out the congruence classes $1\bmod 8$ and $ 5\bmod 8$.  Everything else is the same, until we reach the inner sum over $0<k\le X/tl^2$, where the summand now is one of the characters modulo $8$.  (The question of which specific character modulo $8$ depends on which of the subcases of sign of $f$ and residue class modulo $8$ we are considering.)  Summing this character over consecutive integers is bounded (by $1$ in fact), so we get in case $p=2$ the better bound
\[
\sum_{t|D}\sum_{l<X^{1/2}}1=\tau(D)X^{1/2}.
\]

\end{proof}

\section{Explicit Formula for the Dedekind zeta function}\label{S:Appendix}

For comparison purposes, it will be convenient to have at hand the Explicit Formula in the case of $\zeta(s)L(s,\chi_{-D})$.  For consistency we will use the  scale $\log(\sqrt{D}X)/2\pi$ for the zeros.  (Of course the $X$ is here meaningless; we include it only to be able to relate results to the previous sections.)
\begin{nonumtheorem} Let $g$ be an even Schwartz test
function such that $\hg$ has compact support.    We have 
{\allowdisplaybreaks
\begin{multline*}
2 \sum_{k=1}^\infty \sum_p \frac{(1+\chi_{-D}(p)^k)\log
(p)}{p^{k/2}\log(\sqrt{D}X)}\ \hg\left(\frac{\log (p^k)}{\log(\sqrt{D}X)}\right)\\
+\sum_{\gamma}
g\left(\gamma \frac{\log(\sqrt{D}X)}{2\pi}\right)=\frac{\log\left(D/\pi^2\right)}{\log(\sqrt{D}X)}\intii g(\tau)d\tau\\
2g\left(i/2\frac{\log(\sqrt{D}X)}{2\pi}\right)-2g\left(i(1/2-\delta)\frac{\log(\sqrt{D}X)}{2\pi}\right)
+\\
\frac{1}{\log(\sqrt{D}X)}\intii g(\tau)\tre\Bigg[
\frac{\Gamma'}{\Gamma}\left(\frac14+\frac{i\pi
\tau}{\log(\sqrt{D} X)}\right)+
\frac{\Gamma'}{\Gamma}\left(\frac34+\frac{i\pi
\tau}{\log(\sqrt{D} X)}\right)
\Bigg]
d\tau.
\end{multline*}
}
\end{nonumtheorem}

The contributions of the pole of $\zeta(s)$ at $s=0,1$ and the Landau-Siegel zero of $L(s,\chi_{-D})$ at $s=\delta,1-\delta$ appear on the right side above.
The results of \S \ref{S:2} carry over exactly for the Gamma factors.
The sum over primes is again separated into odd and even terms, and the even terms are exactly (as there is no $f$ contribution)  the previously seen
\[
S_{{\rm even};1}  =  \frac{-2}{\log(\sqrt{D}X)}\sum_{n=1}^\infty \frac{\Lambda_D(n)}{n} \ \hg\left(2\frac{\log
n}{\log (\sqrt{D}X)}\right).
\]
The corresponding results of \S \ref{S:even} carry over exactly.  We rearrange to isolate the sum over odd powers of primes, and the zeros other than the Landau-Siegel zero.  The point here (and the reason we included the $X$ scaling factor) is that the sum over odd powers of primes is exactly as before, missing only the sum over twists $f$.  If we assume Hypothesis H, and that $g$ is positive, we can then estimate the sum over the odd powers.  This gives 
{\allowdisplaybreaks
\begin{multline}\label{Eq:17}
2\sum_{l=0}^\infty \sum_p \frac{\lambda(p)\log
p}{p^{(2l+1)/2}\log (\sqrt{D}X)}
 \hg\left(\frac{\log p^{2l+1}}{\log (\sqrt{D}X)}\right)\\
 +\sum_{\gamma}
g\left(\gamma \frac{\log(\sqrt{D}X)}{2\pi}\right)=
 \frac{\log\left(D/\pi^2\right)}{\log(\sqrt{D}X)}\hg(0)-g(0) \\
 +\frac{1}{\log(\sqrt{D} X)}\intii g(\tau)\tre\Bigg[
\frac{\Gamma'}{\Gamma}\left(\frac14+\frac{i\pi
\tau}{\log(\sqrt{D} X)}\right)+
\frac{\Gamma'}{\Gamma}\left(\frac34+\frac{i\pi
\tau}{\log(\sqrt{D} X)}\right)
\Bigg]
d\tau\\
 + \frac{2}{\log(\sqrt{D}X)}\intii g(\tau)\tre\left[
\frac{\zeta^\prime}{\zeta}\left(1+\frac{4\pi i \tau}{\log(\sqrt{D}X)}\right) +
\frac{L^\prime}{L}\left(1+\frac{4\pi i \tau}{\log(\sqrt{D}X)},1_D\right)\right]d\tau\\
+2g\left(i/2\frac{\log(\sqrt{D}X)}{2\pi}\right)-2g\left(i(1/2-\delta)\frac{\log(\sqrt{D}X)}{2\pi}\right).
\end{multline}
}
Via the results of Lemma \ref{L:gammaest} and Lemma \ref{L:zetaest}, the two integrals on the right side above are
\[
2\left(C-\log(8)+\sum_{q|D}\frac{\log(q)}{q+1}\right)\frac{\hg(0)}{\log(\sqrt{D}X)}+O\left(\frac{1}{\log(\sqrt{D}X)^2}\right).
\]
The terms arising from the pole and the Landau-Siegel zero can be expressed as
\begin{multline*}
2g\left(i/2\frac{\log(\sqrt{D}X)}{2\pi}\right)-2g\left(i(1/2-\delta)\frac{\log(\sqrt{D}X)}{2\pi}\right)=\\
2\int_{-\sigma}^\sigma\hg(u)\left(\exp(-1/2\log(\sqrt{D}X)u)-\exp(-(1/2-\delta)\log(\sqrt{D}X)u)\right)\, du.
\end{multline*}
We bound $\hg(u)$ by $\max \hg$, and compute the integral of the exponentials.  The contribution of the endpoint $+\sigma$ tends to $0$ as $\sqrt{D}X\to\infty$ and we neglect it, to get the bound
\begin{align*}
\ll &\max \hg\cdot \frac{ (\sqrt{D} X)^{\sigma /2}}{\log(\sqrt{D} X)}\left(1
-\frac{ \exp(-\delta \log(\sqrt{D}X) \sigma )}{1-2 \delta  }\right)\\
=&\max \hg\cdot\frac{ (\sqrt{D}X)^{\sigma /2} }{  \log(\sqrt{D}X)}\cdot\left(\left(\sigma \log(\sqrt{D}X)  -2\right)\delta+O(\delta^2)\right)\\
\sim &\sigma \max \hg\cdot  (\sqrt{D}X)^{\sigma /2} \cdot \delta.
\end{align*}

To summarize, (\ref{Eq:17})  looks like
\begin{multline*}
\frac{\log\left(D/\pi^2\right)}{\log(\sqrt{D}X)}\cdot \hg(0)-g(0)+ O\left(\sigma \max \hg\cdot(\sqrt{D}X)^{\sigma /2} \cdot \delta\right)+\\
2\left(C-\log(8)+\sum_{q|D}\frac{\log(q)}{q+1}\right)\frac{\hg(0)}{\log(\sqrt{D}X)}+O\left(\frac{1}{\log(\sqrt{D}X)^2}\right).
\end{multline*}
Note that
\[
\frac{\log\left(D/\pi^2\right)}{\log(\sqrt{D}X)}=\frac{1+\log_D(\pi^2)}{1/2+\log_D(X)}=2+O(\log_D(X)).
\]

\section{Appendix: Notes towards Hypothesis H}\label{S:Appendix2}

In this appendix we adapt some ideas of Yoshida \cite{Yoshida}, to indicate how a Landau-Siegel zero of $L(s,\chi_{-D})$ (or equivalently, the lacunarity of the sequence $\lambda(p)=1+\chi_{-D}(p)$) would tend to push low-lying complex zeros of $L(s,\psi\otimes\chi_f)$ towards the critical line.  

For our test function pair $g$, $\hg$ we can also define 
\[
\mathcal M(\hg)(s)=\int_{-\infty}^\infty \hg(u) \exp((s-1/2)2\pi u) du,
\]
so that
\[
 g(t) =\mathcal M(\hg)(1/2+it)
\]
is the inverse Fourier transform.
To ease notation, it will be convenient to think of $\hg=h$ as the original function, $g=\hh$ as the transform.    For generic $h$, Yoshida denotes $\check{h}(x)=h(-x)$, and $\tilde{h}(x)=\overline{h(-x)}$.  Convolution as usual is defined by
\[
h_1*h_2(x)=\int_{-\infty}^{\infty} h_1(x-y)h_2(y)dy.
\]
For $\rho$ in $\mathbb C$, let $h_\rho(x)$ denote $h(x)\exp(-2\pi \rho x)$.
One easily verifies that
\begin{gather}
\mathcal M(\check{h})(s)=\mathcal M(h)(1-s),\notag \\
\mathcal M(\tilde{h})(s)=\overline{\mathcal M(h)}(1-\overline{s}),\notag \\
\mathcal M(h_1*h_2)(s)= \mathcal M(h_1)(s)\mathcal M(h_2)(s),\label{Eq:y1}\\
\mathcal M(h_\rho)(s)=\mathcal M(h)(s-\rho).\label{Eq:translate}
\end{gather}

Suppose now that $\rho_0$ is a complex zero of $L(s,\psi\otimes\chi_f)$ which is off the critical line.  Thus
\[
\rho_0,\quad 1-\rho_0,\quad \overline{\rho_0},\quad 1-\overline{\rho_0}
\]
are all distinct.
Choose test functions $h_0$ and $h_0^\prime$ so that
\begin{gather*}
\mathcal M(h_0)(\overline{\rho_0})=\mathcal M(h_0)(1-\rho_0)=\mathcal M(h_0)(1-\overline{\rho_0})=0,\\
\mathcal M(h_0^\prime)(\rho_0)=\mathcal M(h_0^\prime)(1-\rho_0)=\mathcal M(h_0^\prime)(1-\overline{\rho_0})=0,\\
\mathcal M(h_0)(1/2)=0=\mathcal M(h_0^\prime)(1/2).
\end{gather*}
(We can take any choice of test function with a zero at some point, and use (\ref{Eq:translate}) to shift that zero to an arbitrary point.  A quadruple convolution and (\ref{Eq:y1}) will then force four zeros.\ \ 
Normalize $h_0$ and $h_0^\prime$ so that
\[
\mathcal M(h_0)(\rho_0)=1=\mathcal M(h_0^\prime)(\overline{\rho_0}).
\]
(We would like to choose \lq bump functions\rq\  $\mathcal M(h_0)$ and $\mathcal M(h_0^\prime)$ with their mass concentrated at $\rho_0$ and $\overline{\rho_0}$, more on this later.)\ \   Define
\[
h=h_0-h_0^\prime+\check{h}_0-\check{h}^\prime_0,
\]
so $h$ is even.  Our test function will be $h*\tilde{h}$, which is also even.  Define
\[
\Phi(s)=\mathcal M(h*\tilde{h})(s)= M(h)(s)\cdot \overline{M(h)}(1-\overline{s})
\]
by (\ref{Eq:y1}).  Thus
\[
\Phi(\rho_0)=(1-0+0-0)\cdot(0-0+0-1)=-1,
\]
and similarly $\Phi=-1$ at $\overline{\rho_0}$, at $1-\rho_0$, and at $1-\overline{\rho_0}$.  (On the other hand, for any $\rho$ \emph{on} the critical line, $1-\overline{\rho}=\rho$ and thus $\Phi(\rho)=|\mathcal M(h)(\rho)|^2\ge0$.  But if $\Phi$ is a \lq concentrated\rq\  enough bump, these will not contribute much.)

Observe that
\[
h*\tilde{h}(0)=\int_{-\infty}^{\infty} h(-y)\tilde{h}(y)dy=\int_{-\infty}^{\infty} |h(y)|^2dy> 0,
\]
as $h$ is even.  
We see that
\[
h*\tilde{h}(x)=\int_{-\infty}^{\infty} h(y-x/2)\overline{h}(y+x/2)dy=o(1)\quad\text{for}\quad x\gg0,
\]
since $h$ has rapid decay, and thus the large primes contribute little.  On the other hand, by the lacunarity of $\lambda(p)$, the small primes should not contribute much either.  
The explicit formula looks like
\begin{multline*}
-4+\text{small error from zeros on critical line}=\\
(h*\tilde{h}(0)> 0)+\text{small error from primes},
\end{multline*}
contradiction.  
\begin{remark} Where do we use that $\rho_0$ is \lq low-lying\rq? The point is that the support of the original test function $\alpha_0$ is compact, but the more convolutions we form, the more the support is spread out.  The number of convolutions $M(\rho_0)$ depends on the density of zeros near $\rho_0$, which in turn depends on the height $\tim(\rho_0)$ of $\rho_0$.  But if the support of $\alpha$ becomes too large, the lacunary nature of $\lambda(p)$ disappears.
\end{remark}

\begin{remark}
This is not quite the right test function.  Since $\alpha$ is $1$ at $\rho_0$ and small elsewhere, its $L^2$ norm is small, and so is the $L^2$ norm of the transform by Plancherel.  Thus we end up with \lq small negative $=$ small positive, thus (small) contradiction.\rq\ \ It would be better to renormalize so that, perhaps, the $L^2$ norm is $1$.
\end{remark}

\end{document}